\title{On Riemann type relations for theta functions on bounded symmetric domains of type $I$}
\author{Atsuhira Nagano}

\documentclass[10pt]{article}
\usepackage{mathrsfs,bm,graphics,amsfonts,amsthm,amssymb,amsmath,amscd,booktabs}
\usepackage{url}
\usepackage{longtable}
\usepackage[dvipdfmx]{graphicx,xcolor}
\usepackage{tikz}
\DeclareFontFamily{U}{mathx}{}
\DeclareFontShape{U}{mathx}{m}{n}{ <-> mathx10 }{}
\DeclareSymbolFont{mathx}{U}{mathx}{m}{n}
\DeclareFontSubstitution{U}{mathx}{m}{n}
\DeclareMathAccent{\widecheck}{0}{mathx}{"71}

\def\bigzerou{\smash{\lower1.7ex\hbox{\b 0}}}
\newtheorem{thm}{Theorem}[section]

\newtheorem{lem}{Lemma}[section]
\newtheorem{prop}{Proposition}[section]
\newtheorem{rem}{Remark}[section]

\def\comment#1{{ }}
\makeatletter
  
      \@addtoreset{equation}{section}
 \usepackage{fancybox}
\begin{document}
\maketitle
\setlength{\baselineskip}{13 pt}
 \renewcommand{\thefootnote}{\fnsymbol{footnote}}

\footnote[0]{Keywords:  theta functions ; symmetric domains  }
\footnote[0]{Mathematics Subject Classification 2020:  Primary 11F27 ; Secondary 32M15, 14K25}
%\footnote[0]{Running head: ...}
\setlength{\baselineskip}{14 pt}

\begin{abstract}
We provide a practical technique to obtain plenty of algebraic relations for theta functions on the bounded symmetric domains of type $I$.
In our framework, each theta relation is controlled by combinatorial properties of a pair $(T,P)$ of a regular matrix $T$ over an imaginary quadratic field and a positive-definite Hermitian matrix $P$ over the complex number field.
\end{abstract}

\section*{Introduction}
Theta functions are special functions which play significant roles in number theory, geometry, combinatorics and mathematical physics.
The $g^2$-dimensional bounded symmetric domain $\mathbb{H}_I^{(g)}$ of type $I$ is the homogeneous space of the unitary group $U(g,g)$ (see (\ref{SymmI})).
In this paper,
we will present  a practical technique
to obtain  plenty of algebraic relations of theta functions on $\mathbb{H}_I^{(g)}$.
Relations obtained by our technique
can be recognized as natural counterparts of the classical Riemann theta relations for the Jacobi theta functions (see (\ref{JacobiTheta}) and (\ref{RiemannOri})). 
For this reason, in this paper,
 we will call our results
the Riemann type relations for the theta functions on $\mathbb{H}_I^{(g)}$.

Let us recall  classical Riemann theta relations.
The Jacobi theta functions are defined as
\begin{align}\label{JacobiTheta}
\mathbb{C} \times \mathbb{H}\ni (z,\tau)
\mapsto
\vartheta_{ab} (z,\tau)=\sum_{n\in \mathbb{Z}} \exp\left( \pi \sqrt{-1} \left(n+\frac{a}{2}\right)^2 \tau + 2\pi \sqrt{-1} \left(n+\frac{a}{2}\right)\left(z+\frac{b}{2}\right) \right)\in \mathbb{C},
\end{align}
where $\mathbb{H}=\{\tau \in \mathbb{C} \mid {\rm Im}(\tau) >0\} $ and $a,b \in \{0,1\}$.
Classical Riemann theta relations are non-trivial quartic relations
which the theta functions (\ref{JacobiTheta}) satisfy.
One can find them in \cite{Mu} Chap.I Sect.5 (see also Section 1.2 of this paper).
Here, let us give a simple example:
\begin{align}\label{RiemannOri}
&2 \vartheta_{00} (x',\tau) \vartheta_{00} (y',\tau) \vartheta_{00} (u',\tau) \vartheta_{00} (v',\tau) \notag\\
&=\vartheta_{00} (x,\tau) \vartheta_{00} (y,\tau) \vartheta_{00} (u,\tau) \vartheta_{00} (v,\tau) 
+\vartheta_{10} (x,\tau) \vartheta_{10} (y,\tau) \vartheta_{10} (u,\tau) \vartheta_{10} (v,\tau) \notag\\
&+\vartheta_{01} (x,\tau) \vartheta_{01} (y,\tau) \vartheta_{01} (u,\tau) \vartheta_{01} (v,\tau)
+\vartheta_{11} (x,\tau) \vartheta_{11} (y,\tau) \vartheta_{11} (u,\tau) \vartheta_{11} (v,\tau), 
\end{align}
where 
$x',y',u',v'$ are complex variables related to $x,y,u,v$ by the orthogonal matrix 
\begin{align}\label{TOri}
T_0=\frac{1}{2} \begin{pmatrix} 1 & 1 & 1 &1 \\  1 & -1& 1 &-1 \\  1& 1 & -1& -1 \\  1& -1& -1 &1  \end{pmatrix}
\end{align}
as ${}^t (x',y',u',v') = T_0 {}^t (x,y,u,v)$.
This example of classical Riemann theta relations has various applications.
\begin{itemize}
\item[\text{[a]}]
Let us recall  a simple and notable example.
Just putting $x=y=u=v=0 $ to (\ref{RiemannOri}), we immediately obtain 
$$
\vartheta_{00}^4(\tau) = \vartheta_{10}^4(\tau) +\vartheta_{01}^4(\tau),
$$
where $\vartheta_{ab}(\tau)=\vartheta_{ab}(0,\tau).$
This relation is  called the Jacobi theta identity.
This is important to study elliptic modular forms for the principal congruence subgroup in ${\rm SL}(2,\mathbb{Z})$ of level $2$.
\end{itemize}

The Jacobi theta functions satisfy other interesting formulas.
\begin{itemize}
\item[\text{[b]}] 
There are ``half formulas'' for the Jacobi theta functions (\ref{JacobiTheta}). 
For detail, see (\ref{Jacobihalf-angle}).
They play an important role in the study of the arithmetic-geometric mean by Gauss (see \cite{C} ; see also the beginning of Section 2.1).
\end{itemize}

For a ring $R$,
let ${\rm Mat}(g,h;R)$ be the abelian group of $g\times h$ matrices over $R$,
and $R^g$ be the $R$-module of column vectors of size $g$ over $R$.
The well-known Riemann theta function on the $\frac{g(g+1)}{2}$-dimensional Siegel upper half plane
$\mathfrak{S}_g =\{\Omega\in {\rm Mat}(g,g;\mathbb{C}) \mid {}^t\Omega=\Omega, {\rm Im}(\Omega) \text{ is positive definite}\}$
is defined as
\begin{align}\label{RiemannThetaCon}
\mathfrak{S}_g \ni \Omega
\mapsto
\vartheta(\Omega)
=\sum_{n\in \mathbb{Z}^g} \exp \left(\pi \sqrt{-1}  {}^t n \Omega n   \right)  
\in \mathbb{C}.
\end{align}

\begin{itemize}
\item[\text{[c]}] 
One can define Riemann theta functions with characteristics (see (1.14)).
They satisfy various non-trivial homogeneous relations.
For instance,  let us refer to the relation (\ref{RiemannQuad}),
which is   important in algebraic geometry and number theory.
This is essential in studies of abelian varieties,
because  Riemann theta functions  are  associated with  moduli spaces of abelian varieties in nature (for example, see \cite{I} and \cite{K}).  
\end{itemize}

There are various proofs of the classical Riemann relations such as (\ref{RiemannOri}).
To the best of the author's knowledge,
each proof is based on  combinatorial properties derived from the orthogonal matrix $T_0$ of  (\ref{TOri}).
Additionally,
one can prove the above results [b] and [c] 
on the basis of  combinatorial properties of  suitable matrices $T$, 
which are taken instead of $T_0$,
and the orthogonality of characters
 (for example, see \cite{Mu} Chap.II).
Here, we remark that the replaced matrices $T$ are not necessarily orthogonal matrices.
In the present paper,
we will call
 the algebraic relations suggested in [b] and [c]  Riemann type relations,
 following the naming of (\ref{RiemannOri}).

The main purpose of this paper is
to obtain  generalizations of  the above classical theta relations.
We study  theta functions on
the $g^2$-dimensional bounded symmetric domain 
\begin{align}\label{SymmI}
\mathbb{H}_I^{(g)}=\left\{ W\in {\rm Mat} (g,g; \mathbb{C}) \hspace{1mm} \Bigg| \hspace{1mm}  \frac{W-{}^t \overline{W}}{2\sqrt{-1}} \text{ is a positive-definite Hermitian matrix} \right\}
\end{align}
of type $I$.
This space naturally contains the Siegel upper half plane $\mathfrak{S}_g$.
Moreover,
the unitary group $U(g,g)$ acts on $\mathbb{H}_I^{(g)}$ in a canonical way.
Thus, we have motives to study theta functions on $\mathbb{H}_I^{(g)}$.
In fact,
by using the lattice $\mathfrak{O}_K^g$ given by the ring $\mathfrak{O}_K$ of integers of an imaginary quadratic field $K$,
we have the theta function
\begin{align}\label{ThetaCon}
\mathbb{H}_I^{(g)} \ni W
\mapsto
\Theta (W)
=\sum_{n\in \mathfrak{O}_K^g} \exp \left(\pi \sqrt{-1}  {}^t \overline{n} W n    \right)
\in \mathbb{C}.
\end{align}
Such a theta function is studied in \cite{F}, \cite{DK} or \cite{Ma} (see also (\ref{theta1}) and Remark \ref{RemMatTheta}). 
The main theorem (Theorem \ref{ThmRiemannRel}) of this paper gives  a  systematic way to  obtain various  relations for the theta function  (\ref{ThetaCon}).
We can regard our results as  generalizations of the above mentioned classical results [a], [b] and [c].
Therefore, we call our results Riemann type  relations for the theta function (\ref{ThetaCon}).
More precisely, in our framework,
a pair $(T, P)$ determines an algebraic theta relation. 
Here, $T$ is a matrix over  $K$, 
which is a counterpart of $T_0$ of (\ref{TOri}). 
Note that we do not need to assume that $T$ is an orthogonal matrix or a unitary matrix.
Also, $P$ is a positive-definite Hermitian matrix over $K$.
For detail, see Theorem \ref{ThmRiemannRel}.

Probably, 
the biggest difference between the theta functions (\ref{ThetaCon}) and (\ref{RiemannThetaCon})
is whether we need an imaginary quadratic field $K$ or not.
In fact, the  theta function (\ref{ThetaCon}) depends on the imaginary quadratic field $K$.
However,
 such a substitutability of $K$ is necessary,
when we practically study the theta functions.  
Actually, 
it is well known that the Riemann theta function (\ref{RiemannThetaCon}) is useful to study the moduli space of $\mathscr{K}_{\rm Kum}$,
where $\mathscr{K}_{\rm Kum}$ is the family of Kummer surfaces coming from principally polarized abelian surfaces (see \cite{I}).
\cite{MSY} constructs a family $\mathscr{K}_{{\rm MSY}}$ of $K3$ surfaces which is an extension of $\mathscr{K}_{\rm Kum}$.
By virtue of the theta function  (\ref{ThetaCon}) associated with $K=\mathbb{Q}\left(\sqrt{-1}\right)$,
we are able to study the moduli space of  $\mathscr{K}_{{\rm MSY}}$.
On the other hand,
\cite{NS} studies another family $\mathscr{K}_{NS}$ of $K3$ surfaces. 
Although $\mathscr{K}_{\rm NS}$ also contains $\mathscr{K}_{\rm Kum}$ as a subfamily,
this  is different from $\mathscr{K}_{{\rm MSY}}$.
In order to study the moduli space of $\mathscr{K}_{NS}$,
we need the theta function  (\ref{ThetaCon}) for $K=\mathbb{Q}\left(\sqrt{-3}\right)$ (see Table 1).
\begin{table}[h]
\center
\begin{tabular}{cccccc}
\toprule
$\mathscr{K}_{\rm Kum}$  & $\mathscr{K}_{{\rm MSY}}$ &  $\mathscr{K}_{{\rm NS}}$  \\
 \midrule
Riemann theta of (\ref{RiemannThetaCon})
     &  Theta of (\ref{ThetaCon}) for  $K=\mathbb{Q}\left(\sqrt{-1}\right)$ & Theta of (\ref{ThetaCon}) for  $K=\mathbb{Q}\left(\sqrt{-3}\right)$ \\
  \bottomrule
\end{tabular}
\caption{Families of $K3$ surfaces and related theta functions  }
\end{table}

Algebraic relations of theta functions enable us to obtain interesting results in number theory and geometry.
To the best of the author's knowledge,
there are several deep research related with $\mathscr{K}_{\rm MSY}$ since the publication of \cite{MSY} and \cite{Mat}.
In this case,
the theta function (\ref{ThetaCon}) is  defined by the use of  the  ring of Gaussian integers.
Theta relations for them lead
a study of an arithmetic hyperbolic $3$-manifold
(a generalization of the arithmetic-geometric mean, resp.)
in \cite{MNY}
 (\cite{MT}, resp.).

Thus, there is a necessity to study algebraic relations for  the theta function (\ref{ThetaCon}) on $\mathbb{H}_I^{(g)}$.
 Our main theorem provides a handy tool to obtain non-trivial theta relations.
 It is valid for arbitrary imaginary quadratic fields $K$.
 The author expects that
 our technique is beneficial to obtain plenty of theta relations 
 and
 it accelerates further research of moduli spaces of families of algebraic varieties
 like $\mathscr{K}_{\rm MSY}$ or  $\mathscr{K}_{{\rm NS}}$.

The contents of this paper are as follows.
In Section 1, 
we introduce a new theta function $\Theta^P\begin{bmatrix}A_0 \\ B_0 \end{bmatrix}(W)$,
where $A_0,B_0\in {\rm Mat}(g,h;K)$ and $P\in {\rm Mat}(h,h;\mathbb{C})$ is a positive-definite Hermitian matrix (see (\ref{ThetaAB})).
We prove an algebraic relation for $\Theta^P\begin{bmatrix}A_0 \\ B_0 \end{bmatrix}(W)$ (Theorem \ref{ThmRiemannRel}).
This is a natural counterpart of the formula of \cite{Mu} Theorem 6.3 in Chap.II
(see also Section 1.2).
If $P$ is a  matrix over $\mathbb{Q}$,
according to Theorem \ref{ThmQ},
$\Theta^P$ is related to the theta function of (\ref{ThetaCon}).
We will recognize theta relations derived from Theorem \ref{ThmRiemannRel} as generalizations of the classical Riemann type relations  such as (\ref{RiemannOri}), [a], [b] and [c]. 
By taking  suitable pairs $(T,P)$ and applying Theorem \ref{ThmRiemannRel},
we obtain   non-trivial theta relations by  relatively simple calculations.
In order to see practical effects of our main result,
we will apply it in Section 2 for appropriately chosen   $(T,P)$.

\section{Relations for theta functions on $\mathbb{H}_I^{(g)}$}

\subsection{Main result}

Throughout this paper, $K$ denotes the imaginary quadratic field $\mathbb{Q}\left(\sqrt{-d}\right)$ for a positive square-free integer $d$, and $\mathfrak{O}_K=\mathbb{Z}+\delta \mathbb{Z}$ denotes the ring of integers of $K$, where
  \begin{align*} 
\delta=
\begin{cases}
 \sqrt{-d}  \quad  &(-d \not\equiv 1 (\text{mod } 4)), \\
 \frac{1+\sqrt{-d}}{2} \quad  &(-d \equiv 1 (\text{mod } 4)).
\end{cases}
\end{align*}
Note that it holds $|\delta|^2 \in \mathbb{Z}$.

As mentioned in Introduction, 
$\mathbb{H}_I^{(g)}$  in (\ref{SymmI}) stands for the $g^2$-dimensional bounded symmetric domain of type $I$.
Here, we remark that every $W\in \mathbb{H}_I^{(g)}$ can be decomposed into the sum of a Hermitian matrix and $\sqrt{-1}$ times a positive-definite Hermitian matrix:
\begin{align}\label{WDec}
W=\frac{W+{}^t \overline{W}}{2} 
+\sqrt{-1} \frac{W-{}^t \overline{W}}{2 \sqrt{-1}}. 
\end{align}

Let $P \in {\rm Mat} (h,h; \mathbb{C})$ be a positive-definite Hermitian matrix. 
 Let us define a theta function 
\begin{align*}
\Theta^P (W) =
\sum_{N \in {\rm Mat}(g,h;\mathfrak{O}_K) }
\exp \left(\pi \sqrt{-1} {\rm Tr} \left({}^t \overline{N} W N P\right) \right)
\end{align*}
on $\mathbb{H}_I^{(g)}$ associated with the imaginary quadratic field  $K$.
The right hand side 
is an absolutely and uniformly convergent series  on any compact set in $\mathbb{H}_I^{(g)}$.

Additionally, we define the  theta function with characteristics $A_0, B_0 \in {\rm Mat} (g,h;K)$
\begin{align}
\label{ThetaAB}
\Theta^P\begin{bmatrix} A_0 \\ B_0 \end{bmatrix} (W) =\sum_{N\in {\rm Mat}(g,h;\mathfrak{O}_K)} \exp \left(\pi \sqrt{-1} {\rm Tr} \left({}^t \overline{(N+A_0)} W (N+A_0) P\right) +2\pi \sqrt{-1} {\rm Re} {\rm Tr} \left({}^t \overline{(N+A_0)} B_0\right)  \right)
\end{align}
on $\mathbb{H}_I^{(g)}$ 
associated with $K$.

Here, let us briefly see the convergence of (\ref{ThetaAB}) just in case.
Since $P$ is a positive-definite Hermitian matrix,
there is a unitary matrix $U$ such that $\overline{U} P {}^tU=\text{diag}(\gamma_1,\ldots,\gamma_h)$ with $\gamma_j >0$, where $\text{diag}(\gamma_1,\ldots,\gamma_h)$ is the diagonal matrix with diagonal entries $\gamma_1,\ldots, \gamma_h.$
Then, putting $N'=(\nu_1,\ldots,\nu_h) = (N+A_0) {}^t U$,
we have
$
{\rm Tr}  \left({}^t \overline{(N+A_0)} W (N+A_0) P\right) = \sum_{j=1}^h \gamma_j {}^t\overline{\nu_j} W \nu_j,
$
where ${\rm Tr}$ is the matrix trace.
Thereby, from (\ref{WDec}), we obtain
\begin{align*}
&\sum_{N\in {\rm Mat}(g,h;\mathfrak{O}_K)} \Big|\exp \left(\pi \sqrt{-1} {\rm Tr} \left({}^t \overline{(N+A_0)} W (N+A_0) P\right) +2\pi \sqrt{-1} {\rm Re} {\rm Tr} \left({}^t \overline{(N+A_0)} B_0\right)  \right)\Big| \\
&=\prod_{j=1}^h \sum_{\nu_j} \exp\left(-\gamma_j  {}^t\overline{\nu_j} \frac{W-{}^t\overline{W}}{2 \sqrt{-1}}\nu_j \right).
\end{align*}
According to (\ref{SymmI}),
we can compare the right hand side of (\ref{ThetaAB}) with a product of the Gaussian integrals $\int_\mathbb{R} e^{-X^2}dX <\infty$.
Thus, we can check the convergence of (\ref{ThetaAB}).

We set
\begin{align*}
\widehat{B}=
\begin{cases}
B  \quad \quad  &(-d \not \equiv 1 \hspace{1mm} (\text{mod } 4)), \\
 2B \quad \quad &(-d  \equiv 1 \hspace{1mm} (\text{mod } 4)), \\
\end{cases}
\end{align*}
for any $B\in {\rm Mat}(g,h; K)$.

Using the above notations, we have the following theorem.

\begin{thm} (Riemann type relation of theta functions on $\mathbb{H}_I^{(g)}$) \label{ThmRiemannRel}
Let $K$ be an imaginary quadratic field.
Let $A_0,B_0\in {\rm Mat} (g,h; K)$.
Suppose $(T,P)$ is a pair of matrices satisfying
\begin{itemize}
\item[(i)]
$T \in {\rm GL}_h (K)$,
\item[(ii)]
 $P \in {\rm Mat}(h,h; \mathbb{C})$ is a positive-definite Hermitian matrix.
\end{itemize}
Also, let $Q$ be a positive-definite Hermitian matrix given by
$Q= {}^t \overline{T} P T$.
Then, the theta function (\ref{ThetaAB}) associated with $K$ satisfies the relation
\begin{align}\label{RiemannRel}
\Theta^{Q}\begin{bmatrix} A_0 {}^t \overline{T}^{-1} \\ B_0 T \end{bmatrix} (W) 
= \frac{1}{{}^\sharp G_2} \sum_{B\in G_2} \sum_{A\in G_1} \exp\left( -2\pi \sqrt{-1} {\rm Re} {\rm Tr} \left({}^t \overline{A_0} \widehat{B} \right) \right) \cdot \Theta^P \begin{bmatrix} A_0+A\\ B_0+\widehat{B} \end{bmatrix} (W).
\end{align}
Here, $G_1$ and $G_2$ are finite abelian groups defined as
\begin{itemize}
\item $G_1={\rm Mat}(g,h; \mathfrak{O}_K) {}^t \overline{T}/ \left({\rm Mat}(g,h; \mathfrak{O}_K) {}^t \overline{T} \cap {\rm Mat}(g,h; \mathfrak{O}_K) \right) $,
\item $G_2= {\rm Mat}(g,h; \mathfrak{O}_K) T^{-1}/\left({\rm Mat}(g,h; \mathfrak{O}_K) T^{-1} \cap {\rm Mat}(g,h; \mathfrak{O}_K) \right) $,
\end{itemize}
and ${}^\sharp G_2$ denotes the order of $G_2$.
\end{thm}

\begin{proof}
By the definition, we have
\begin{align}\label{LHS}
\notag
&\Theta^{Q}\begin{bmatrix} A_0 {}^t\overline{T}^{-1} \\ B_0 T \end{bmatrix}(W) \notag\\
&=\sum_{N\in {\rm Mat}(g,h;\mathfrak{O}_K)} 
\exp\bigg(\pi \sqrt{-1} {\rm Tr}\left({}^t \overline{\left(N+A_0 {}^t \overline{T}^{-1}\right)} W \left(N+A_0 {}^t \overline{T}^{-1}\right)
 \left({}^t\overline{T} P T \right) \right) \notag\\
 &\hspace{5cm}  +2\pi \sqrt{-1} {\rm Re} {\rm Tr} \left({}^t \overline{\left(N+A_0 {}^t \overline{T}^{-1}\right)}  B_0 T \right)   \bigg) \notag\\
&\notag
=\sum_{N\in {\rm Mat}(g,h;\mathfrak{O}_K)} \exp\Big( \pi \sqrt{-1} {\rm Tr} \left( {}^t \overline{(N {}^t \overline{T} + A_0)} W (N {}^t\overline{T} +A_0) P \right)  + 2\pi \sqrt{-1} {\rm Re} {\rm Tr} \left({}^t \overline{\left(N {}^t \overline{T} +A_0 \right)}  B_0\right) \Big) \notag \\
&=\sum_{M\in {\rm Mat}(g,h;\mathfrak{O}_K) {}^t \overline{T}} \exp\Big( \pi\sqrt{-1} {\rm Tr} \left({}^t \overline{\left(M+A_0 \right)} W \left( M +A_0 \right) P\right) + 2\pi \sqrt{-1} {\rm Re} {\rm Tr}\left({}^t \overline{\left(M+A_0 \right)} B_0 \right) \Big).
\end{align}
On the other hand,
since $A\in G_1$, we have
\begin{align}\label{RHSp}
& \sum_{B\in G_2} \sum_{A\in G_1}
\exp\left( -2\pi \sqrt{-1} {\rm Re} {\rm Tr} \left({}^t \overline{A_0} \widehat{B} \right) \right) 
\cdot \Theta^P\begin{bmatrix} A_0+A \\ B_0+\widehat{B} \end{bmatrix} (W) \notag\\
&=\sum_{B\in G_2} \sum_{A\in G_1} \sum_{N\in {\rm Mat}(g,h;\mathfrak{O}_K) }
\exp\Big( \pi\sqrt{-1} {\rm Tr}\left({}^t\overline{(N+A_0+A)} W (N+A_0+A) P \right) \notag\\
&\hspace{5cm} +2\pi \sqrt{-1} {\rm Re} {\rm Tr} \left({}^t\overline{(N+A_0+A)} \left(B_0+\widehat{B} \right) \right)-2\pi \sqrt{-1} {\rm Re} {\rm Tr} \left({}^t \overline{A_0} \widehat{B} \right)  \Big) \notag \\
&=\sum_{B\in G_2}  
\hspace{3mm} \sum_{M\in {\rm Mat}(g,h;\mathfrak{O}_K) {}^t \overline{T} + {\rm Mat}(g,h;\mathfrak{O}_K)  } 
\exp\Big( \pi\sqrt{-1} {\rm Tr}\left( {}^t\overline{\left(M+A_0\right)} W \left(M+A_0 \right) P \right)  \notag\\
&\hspace{6.9cm}+ 2\pi \sqrt{-1} {\rm Re} {\rm Tr}\left({}^t \overline{\left(M+A_0 \right)} B_0 \right)+ 2\pi \sqrt{-1} {\rm Re} {\rm Tr} \left( {}^t \overline{M} \widehat{B} \right) \Big). 
\end{align}
Now, let $M\in {\rm Mat}(g,h;\mathfrak{O}_K) {}^t \overline{T} + {\rm Mat}(g,h;\mathfrak{O}_K) $.
Then,  
${\rm Re} {\rm Tr} \left( {}^t\overline{M} \widehat{B} \right)\in \mathbb{Z}$ for every $B\in G_2$
if and only if
$M\in  {\rm Mat}(g,h;\mathfrak{O}_K) {}^t \overline{T}$ (see Lemma \ref{LemRingInt} below).
We note that the correspondence
$$
B \mapsto \exp\left(2\pi \sqrt{-1} {\rm Re} {\rm Tr} \left({}^t \overline{M} \widehat{B}\right)\right)
$$
gives a character of  the finite abelian group $G_2$.
In fact,
letting $C=L T^{-1} \in {\rm Mat}(g,h;\mathfrak{O}_K)T^{-1} \cap {\rm Mat}(g,h;\mathfrak{O}_K)$ with $L\in {\rm Mat}(g,h;\mathfrak{O}_K)$,
we have
\begin{align*}
\exp\left( 2 \pi \sqrt{-1} {\rm Re} {\rm Tr}  \left( {}^t\overline{ \left( M_1 {}^t \overline{T}+ M_2 \right) } \widehat{C} \right) \right)
=\exp\left( 2 \pi \sqrt{-1} \left({\rm Re} {\rm Tr}  \left( {}^t\overline{  M_1 }  \widehat{L}\right)  + {\rm Re} {\rm Tr}  \left( {}^t\overline{  M_2 }  \widehat{C}\right)  \right)\right) =1\cdot 1 =1
\end{align*}
for any $M_1,M_2 \in{\rm Mat}(g,h;\mathfrak{O}_K).$
Therefore, it holds
\begin{align}\label{SCh}
\sum_{B\in G_2} \exp\left(2\pi \sqrt{-1} {\rm Re} {\rm Tr} \left(M \widehat{B}\right)\right) =
\begin{cases}
 {}^\sharp G_2  \quad &(M\in {\rm Mat}(g,h;\mathfrak{O}_K) {}^t \overline{T}),\\
 \hspace{2mm} 0 \quad  &(\text{otherwise}).
\end{cases}
\end{align}
Hence, by (\ref{RHSp}) and (\ref{SCh}), we obtain
\begin{align}\label{RHS}
&\sum_{B\in G_2} \sum_{A\in G_1} \exp\left( -2\pi \sqrt{-1} {\rm Re}{\rm Tr} ({}^t\overline{A_0} \widehat{B}) \right) \Theta^P \begin{bmatrix} A_0+A \\ B_0+\widehat{B} \end{bmatrix} (W) \notag\\
&=  {}^\sharp G_2 \sum_{M\in {\rm Mat}(g,h;\mathfrak{O}_K) {}^t \overline{T}}
 \exp\left( \pi\sqrt{-1} {\rm Tr} \left({}^t \overline{\left(M+A_0 \right)} W \left( M +A_0 \right) P\right)
 + 2\pi \sqrt{-1} {\rm Re} {\rm Tr}\left({}^t \overline{\left(M+A_0 \right)} B_0 \right) \right).
\end{align}
By comparing (\ref{LHS}) with (\ref{RHS}),
we have (\ref{RiemannRel}).
\end{proof}

\begin{lem}\label{LemRingInt}
Suppose that $M\in {\rm Mat}(g,h; K)  $. 
Then,
${\rm Re} {\rm Tr} \left( {}^t\overline{M} \widehat{B} \right)\in \mathbb{Z}$ for every $B\in {\rm Mat}(g,h;\mathfrak{O}_K) T^{-1}$
if and only if
$M\in  {\rm Mat}(g,h;\mathfrak{O}_K) {}^t \overline{T}$.
 \end{lem}

\begin{proof}
Suppose that ${\rm Re}{\rm Tr} \left({}^t\overline{M} \widehat{B} \right) \in \mathbb{Z}$ for all $B\in {\rm Mat}(g,h;\mathfrak{O}_K) T^{-1}$. 
Let us take $B= \ell E_{ab}  T^{-1}$,
where $\ell \in \mathfrak{O}_K$ and  $E_{ab}$ $(1\leq a \leq g, 1\leq b \leq h)$ is the matrix unit.
Setting
$S =M {}^t \overline{T}^{-1}=(s_{jk}) \in {\rm Mat} (g,h; K)$,
we have
\begin{align}\label{ReTrells}
{\rm Re} {\rm Tr} \left({}^t \overline{M}\widehat{B}\right) = 
\begin{cases}
 {\rm Re} (\ell\overline{s_{ab}}) \quad  &(-d \not\equiv 1 \hspace{1mm} (\text{mod }4)), \\
2{\rm Re} (\ell\overline{s_{ab}})  \quad  &(-d \equiv 1 \hspace{1mm} (\text{mod }4)). 
\end{cases}
\end{align}
By the assumption,
together with (\ref{ReTrells}), 
it follows that $s_{ab}\in \mathfrak{O}_K$ for all $a$ and $b$.
Hence, 
$M=S {}^t \overline{T} \in {\rm Mat} (g,h; \mathfrak{O}_K) {}^t \overline{T}$.

Conversely,
by setting $M=S{}^t\overline{T}$ 
and
$B=L  T^{-1}$ 
for $S, L\in {\rm Mat} (g,h ; \mathfrak{O}_K)$,
we obtain
\begin{align*}
{\rm Re}{\rm Tr}\left({}^t \overline{M} \widehat{B} \right) 
=\begin{cases}
{\rm Re} {\rm Tr} \left( {}^t \overline{S} L \right)  \quad &(-d \not\equiv 1 \hspace{1mm} (\text{mod }4)),\\
2 {\rm Re} {\rm Tr} \left( {}^t \overline{S} L \right) \quad  &(-d \equiv 1 \hspace{1mm} (\text{mod }4)).
\end{cases} 
\end{align*}
Since $\mathfrak{O}_K$  is a ring,
${\rm Re}{\rm Tr}\left({}^t \overline{M} \widehat{B} \right)$ is a rational integer.
\end{proof}

\begin{rem}
The groups $G_1$ and $G_2$ of the theorem are finite,
because every entry $t$ of $T\in {\rm Mat}(h,h;K)$ has an expression $t=\frac{\nu}{n}$ with $\nu \in \mathfrak{O}_K$ and $n\in \mathbb{Z}$.
This is why we assume that $T$ is a matrix over $K$.
\end{rem}

\subsection{Backgrounds}

For $a, b\in K^g$,
we have the holomorphic theta function
\begin{align}\label{theta1}
\Theta\begin{bmatrix} a \\ b  \end{bmatrix} (W)
=\sum_{n\in \mathfrak{O}_K^g} \exp \left(\pi \sqrt{-1}  \left({}^t \overline{(n+a)} W (n+a) \right) +2\pi \sqrt{-1} {\rm Re}  \left({}^t \overline{(n+a)} b \right)  \right)
\end{align}
on $\mathbb{H}_I^{(g)}$
associated with $K$.
By the definition, this function satisfies
\begin{align}\label{ThetaRed}
\Theta\begin{bmatrix} ua \\ ub \end{bmatrix} (W) = 
\Theta\begin{bmatrix} a+\nu \\ b \end{bmatrix} (W) =
\Theta\begin{bmatrix} a \\ b \end{bmatrix} (W),
\end{align}
for every $\nu \in \mathfrak{O}_K^g$  and  every unit $u \in \mathfrak{O}_K^\times$.

\begin{rem}\label{RemMatTheta}
The theta function $\Theta\begin{bmatrix} a \\ b \end{bmatrix} (W)$ of (\ref{theta1}) is essentially equal to the theta function defined in [Ma2] as
\begin{align}\label{thetaM}
\widecheck{\Theta}\begin{bmatrix} a \\ b \end{bmatrix} (W)
=
\begin{cases}
\sum_{n\in \mathfrak{O}_K^g} \exp \left(\pi \sqrt{-1}  \left({}^t \overline{(n+a)} W (n+a) \right) +2\pi \sqrt{-1} {\rm Re}  \left({}^t \overline{n} b \right)  \right) &\quad (-d\not\equiv 1  \hspace{1mm}({\rm mod } \hspace{1mm} 4)),\\
\sum_{n\in \mathfrak{O}_K^g} \exp \left(\pi \sqrt{-1}  \left({}^t \overline{(n+a)} 2W (n+a) \right) +4\pi \sqrt{-1} {\rm Re}  \left({}^t \overline{n} b \right)  \right) &\quad (-d\equiv 1 \hspace{1mm} ({\rm mod } \hspace{1mm}  4)).
\end{cases}
\end{align}
If $-d\not\equiv 1 \hspace{1mm} ({\rm mod } \hspace{1mm}  4)$ ($-d\equiv 1 \hspace{1mm} ({\rm mod } \hspace{1mm}  4)$, resp.),
the difference between
$\widecheck{\Theta}\begin{bmatrix} a \\ b \end{bmatrix}(W)$ and 
our theta function 
$\Theta\begin{bmatrix}a \\ b \end{bmatrix} (W)$ 
($\Theta\begin{bmatrix}a \\ 2b \end{bmatrix} (2W)$, resp.) 
is the factor $e^{2\pi \sqrt{-1} ab}$ ($e^{4\pi \sqrt{-1} ab}$, resp.).
\end{rem}

In the following, let us  see  the reason why we call the equation   (\ref{RiemannRel}) the  Riemann type relation while  retracing the origin.
Let $I_h$ be the identity matrix of size $h$.

\begin{itemize}

\item Let us start with the simplest case.
If $h=1$ and $P=I_1$,
then $\Theta^P\begin{bmatrix}A_0 \\ B_0 \end{bmatrix} (W)$ of (\ref{ThetaAB}) is the same as the theta function of (\ref{theta1}).

\item
If $P=I_h$, 
then $\Theta^P\begin{bmatrix}A_0 \\ B_0 \end{bmatrix} (W)$ has an expression of a product of $h$ theta functions of (\ref{theta1}).
In practice, 
if we suppose $A_0=(a_1,\ldots, a_h)$ and $B_0=(b_1,\ldots,b_h)$ with $a_j,b_j \in K^g$,
by the definition,
we have
\begin{align}\label{ProdTheta1}
&\Theta^{I_h} \begin{bmatrix} A_0 \\ B_0  \end{bmatrix} (W) 
=\sum_{n_1,\ldots,n_h\in \mathfrak{O}_K^g} \exp\left( \pi\sqrt{-1}\left( \sum_{j=1}^h {}^t\overline{(n_j+a_j) } W (n_j+a_j) 
\right)+ 2\pi \sqrt{-1} {\rm Re} \left(\sum_{j=1}^h {}^t\overline{(n_j+a_j) } b_j \right)  \right) \notag\\
&=\prod_{j=1}^h \sum_{n_j \in \mathfrak{O}_K^g} 
\exp\left(\pi \sqrt{-1} \left({}^t \overline{(n_j+a_j)} W (n_j+a_j)\right) 
+2\pi \sqrt{-1} {\rm Re} \left({}^t\overline{(n_j+a_j)} b_j \right)  \right) 
=\prod_{j=1}^h \Theta\begin{bmatrix} a_j \\ b_j \end{bmatrix} (W).
\end{align}
Similarly, 
if $P$ is a diagonal matrix $\text{diag}(\alpha_1,\ldots,\alpha_h) $ with positive rational entries $\alpha_j$, 
we  obtain
\begin{align}\label{ProdTheta2}
\Theta^{\text{diag}(\alpha_1,\ldots,\alpha_h)} \begin{bmatrix} A_0 \\ B_0  \end{bmatrix} (W)  =\prod_{j=1}^h \Theta\begin{bmatrix} a_j \\ b_j \end{bmatrix} (\alpha_j W).
\end{align}

\item Let $\Omega\in \mathfrak{S}_g$ and $z\in \mathbb{C}^g$.
Set
\begin{align}\label{RiemannTheta}
\vartheta\begin{bmatrix}a \\ b \end{bmatrix}(z,\Omega)
=\sum_{n\in \mathbb{Z}^g} \exp \left(\pi \sqrt{-1}  \left({}^t (n+a) \Omega (n+a) \right) +2\pi \sqrt{-1}   \left({}^t (n+a) (z+b) \right)  \right),
\end{align}
where
$a,b\in \mathbb{Q}^g$ (recall (\ref{RiemannOri})). 
This is  the well-known Riemann theta function with characteristics $a,b$.
The theta function (\ref{theta1}) is a natural counterpart of  (\ref{RiemannTheta}) for $z=0\in \mathbb{C}^g$.
One can define the theta functions $\vartheta^P\begin{bmatrix}A_0 \\ B_0 \end{bmatrix}(Z,\Omega)$
for any positive-definite  symmetric matrix $P\in {\rm Mat}(h,h)(\mathbb{R})$ and $Z\in {\rm Mat}(g,h;\mathbb{C})$. 
For $(T,P)$ and $Q$,
where $P$ and $Q$ are positive-definite symmetric matrices satisfying $Q={}^t T P T$ with $T\in {\rm GL}_h(\mathbb{Q})$,
one can obtain a non-trivial relation for $\vartheta^P$ and $\vartheta^Q$ (see \cite{Mu} Chap.II Sect.6).
In fact,
our Theorem \ref{ThmRiemannRel} is a counterpart of that 
to theta functions on $\mathbb{H}_I$.

\item 
When $g=1$,
the Riemann theta function (\ref{RiemannTheta}) is equal to the classical Jacobi theta function (\ref{JacobiTheta}).
In this simple case,
the pair $(T_0, I_4)$ in the above sense
induces  quartic relations of Jacobi theta functions,
where $T_0$ is the orthogonal matrix of (\ref{TOri}).
They coincide with  the classical Riemann theta relations  mentioned in Introduction, like (\ref{RiemannOri}).

\end{itemize}

\subsection{ $\Theta^P$ for $P$ over $\mathbb{Q}$ }

If $P\in {\rm Mat}(h,h ; \mathbb{Q})$ is a positive-definite symmetric matrix,
the theta function $\Theta^P \begin{bmatrix}A_0 \\ B_0 \end{bmatrix}(W)$ of (\ref{ThetaAB}) has an expression in terms of $\Theta\begin{bmatrix} a \\ b\end{bmatrix}(W)$ of (\ref{theta1}) as the following theorem.
This result gives a natural extension of the facts (\ref{ProdTheta1}) and (\ref{ProdTheta2}).

\begin{thm}\label{ThmQ}
Let $K$ be an arbitrary imaginary quadratic field.
Suppose $P \in {\rm Mat}(h,h; \mathbb{Q})$ be a positive-definite symmetric matrix.  
Let $A_0,B_0 \in {\rm Mat}(g,h;K)$.
Then, 
there exist positive rational numbers $\lambda_1,\ldots,\lambda_h$ 
such that
$\Theta^P \begin{bmatrix} A_0 \\ B_0 \end{bmatrix} (W)$ of (\ref{ThetaAB}) associated with $K$
is expressed by a homogeneous polynomial in 
$\Theta\begin{bmatrix} a \\ b \end{bmatrix} (\lambda_j W)$ ($j\in\{1,\ldots, h\}$)
 of degree $h$ over $\mathbb{Q}$.
Here, $\Theta\begin{bmatrix} a \\ b \end{bmatrix} (W)$ is the theta function of (\ref{theta1}).
\end{thm}

\begin{proof}
We suppose that the positive-definite symmetric matrix $P$ is given by
$$
P=\begin{pmatrix}
\mu_1 & R \\
{}^t R & P_{1}
\end{pmatrix},
$$
where $\mu_1 \in \mathbb{Q}, R\in {\rm Mat}(1,h-1;\mathbb{Q})$ and $P_1\in {\rm Mat}(h-1,h-1; \mathbb{Q})$.
Since $P$ is positive definite,
it follows that $\mu_1>0$ and $P_1$ is also positive definite.
In fact, 
we have ${}^t \overline{x} P x>0$ for any non-zero $x\in \mathbb{R}^h$.
So, if we take $x={}^t (x_1,0,\ldots,0) $ ($x={}^t (0,x_2,\ldots,x_{h})$, resp.),
we can see that $\mu_1>0$ ($P_1$ is positive definite, resp.)
In particular, $P_1$ is a nonsingular matrix.
Set 
$$
M=\begin{pmatrix}
1& 0 \\ 
-P_1^{-1} {}^t R& I_{h-1}
\end{pmatrix}
$$
We remark $M\in {\rm Mat} (h,h; \mathbb{Q})$,
because $P$ is a matrix over $\mathbb{Q}$.
Then, it holds
\begin{align}\label{MPMS}
{}^t \overline{M} P M 
=\begin{pmatrix} \mu_1 -R P_1^{-1} {}^t R & 0 \\ 0 & P_1  \end{pmatrix} =: S.
\end{align}
Obviously, $\det(M)=1$ and $M$ is nonsingular.
Setting $K=M^{-1}$,  we have
\begin{align}\label{TSTP}
{}^t \overline{K} S K = P.
\end{align}
Put $\lambda_1=\mu_1 - RP_1^{-1} {}^t R \in \mathbb{Q}.$
Then, it holds $\det (S) = \lambda_1 \det(P_1) = \det (P)$, for $\det(K)=1$.
Since $P_1$ is positive definite,
we have $\lambda_1>0$.
Also, the symmetric matrix $S$ is positive definite.
Thus, especially according to (\ref{TSTP}),
the matrices $K,S$ and $P$ satisfy the condition of Theorem \ref{ThmRiemannRel}.
By virtue of that theorem,
we can obtain an expression of $\Theta^P \begin{bmatrix} A_0 \\ B_0\end{bmatrix} (W)$
given by a linear combination of  theta functions $\Theta^S \begin{bmatrix} A \\ B \end{bmatrix}(W)$ over $\mathbb{Q}$
for appropriate  $A,B \in {\rm Mat}(g,h;K)$.
From (\ref{MPMS}) and the definition (\ref{ThetaAB}) of our theta functions,
each $\Theta^S \begin{bmatrix} A \\ B \end{bmatrix}(W)$ is equal to a product of 
$\Theta\begin{bmatrix} a \\ b \end{bmatrix}(\lambda_1 W)$
and
$\Theta^{P_1} \begin{bmatrix}A_1 \\ B_1 \end{bmatrix}(W)$
for some $a,b\in K^g$ and $A_1,B_1\in {\rm Mat}(g,h-1;K)$.
From the above,
$\Theta^P\begin{bmatrix} A_0 \\ B_0 \end{bmatrix} (W)$ is given by a linear combination of 
$\Theta\begin{bmatrix}a \\ b \end{bmatrix}(\lambda_1 W)\cdot \Theta^{P_1}\begin{bmatrix} A_1 \\ B_1 \end{bmatrix}(W)$ over $\mathbb{Q}$.
Therefore, by mathematical induction,
we have the assertion.
\end{proof}

\begin{rem}
If we consider   cases for general positive-definite Hermitian matrices $P\in {\rm Mat}(h,h;K)$,
instead of    symmetric matrices over $\mathbb{Q}$,
the statement of the above theorem is not always true.
In fact, 
the counterpart of the relation (\ref{MPMS}) does not always hold 
if $P$ is not a matrix over $\mathbb{Q}$.
\end{rem}

The above theorem guarantees that
the theta function $\Theta^P\begin{bmatrix}A \\ B \end{bmatrix}(W)$ of (\ref{ThetaAB})
has a polynomial expression in terms of the ordinary theta functions of (\ref{theta1}),
if $P$ is a  positive-definite  matrix over $\mathbb{Q}$.
Therefore,
such a theta function $\Theta^P\begin{bmatrix}A \\ B \end{bmatrix}(W)$
is useful to obtain non-trivial relations for theta functions  on $\mathbb{H}_I^{(g)}$.

\section{Examples of applications}
Recall that
the classical Riemann theta relation  (\ref{RiemannOri})  can be derived from the combinatorial properties associated with the particular matrix $T_0$ of (\ref{TOri}).
Also, 
by using appropriate matrices $T$ instead of $T_0$,
we can systematically obtain the  Riemann type relations such as [b] and [c] of Introduction.
In the author's opinion,
the most interesting aspect of Theorem \ref{ThmRiemannRel} is the fact that
the theorem enables us to obtain explicit algebraic theta relations by just taking suitable pairs $(T,P)$.
In this section,
we will see how to apply our main theorem to obtain explicit  relations
for the theta functions on $\mathbb{H}_I^{(g)}$
by virtue of appropriate $(T,P)$.

Matsumoto \cite{Ma} obtains non-trivial algebraic relations of theta functions on $\mathbb{H}_I^{(g)}$
for the quadratic fields $K=\mathbb{Q}\left( \sqrt{-1} \right), \mathbb{Q}\left( \sqrt{-2} \right)$ and $\mathbb{Q}\left( \sqrt{-3} \right)$.
He gives formulas in the form
\begin{align}\label{ThetaMat}
\left(\text{a monomial in } \widecheck{\Theta}\begin{bmatrix} a_j \\ b_j \end{bmatrix} (W)\right) = \left(\text{a polynomial in } \widecheck{\Theta}\begin{bmatrix} a_k' \\ b_k' \end{bmatrix} (W)\right).
\end{align}
In fact, 
our main theorem helps us to comprehend the results of \cite{Ma}. 
Namely, by applying Theorem \ref{ThmRiemannRel} to an appropriately chosen unitary matrix $T$ and $P=I_h$,
together with (\ref{ProdTheta2}),
we can acquire each result of \cite{Ma}.
For example, 
let $\widecheck{\Theta}\begin{bmatrix}a \\ b \end{bmatrix}(W)$ be the theta function of (\ref{thetaM}) associated with $K=\mathbb{Q}\left(\sqrt{-1}\right)$.
Then, Theorem 1 of \cite{Ma} is stated as
\begin{align}\label{MatRel}
&2^g \widecheck{\Theta}\begin{bmatrix} a_1 +a_2 \\ b_1+b_2 \end{bmatrix} (W) \cdot \widecheck{\Theta}\begin{bmatrix} a_1 -a_2 \\ b_1-b_2 \end{bmatrix} (W)\notag\\
&=\sum_{e,f\in (\frac{1}{1+\sqrt{-1}}\mathfrak{O}_K/\mathfrak{O}_K)^g} 
\exp\left(2\pi \sqrt{-1} {\rm Re} \left(\left(1+\sqrt{-1} \right) {}^t(b_1+b_2) \overline{f}\right)\right)\notag\\
&\hspace{5cm}
\cdot\widecheck{\Theta}\begin{bmatrix} e+(1+\sqrt{-1})a_1  \\ f +(1+\sqrt{-1}) b_1 \end{bmatrix}(W) \cdot \widecheck{\Theta}\begin{bmatrix} e+(1+\sqrt{-1})a_2 \\ f +(1+\sqrt{-1}) b_2\end{bmatrix}(W), 
\end{align}
where $a_1,a_2,b_1,b_2 \in K^g$.
Now, let us consider the matrices $T,P$ and $Q$ given by
$$T=\frac{1-\sqrt{-1}}{2} \begin{pmatrix} 1 & 1 \\ 1 & -1 \end{pmatrix},
\quad
P=Q=\begin{pmatrix}1 & 0 \\ 0 & 1 \end{pmatrix}.
$$
Also, we set
$A_0=(a_1+a_2,a_1-a_2)$ and $B_0=(b_1+b_2,b_1-b_2)$. 
Now, $T$ is a unitary matrix:
${}^t\overline{T}=T^{-1}=\frac{1+\sqrt{-1}}{2} \begin{pmatrix} 1 & 1 \\ 1 & -1 \end{pmatrix}.$
Then, we can see that
$G_1=G_2=\left(\frac{1}{1+\sqrt{-1}}\mathfrak{O}_K/\mathfrak{O}_K\right)^g.$
Regarding the correspondence between $\widecheck{\Theta}$ and $\Theta$  stated in Remark \ref{RemMatTheta},
we can see that our Theorem \ref{ThmRiemannRel} derives (\ref{MatRel}).

In this section, 
in order to demonstrate availabilities of our Theorem \ref{ThmRiemannRel},
we will obtain  theta relations which are different from (\ref{ThetaMat}).
In particular, the matrices $T$ need not to be  unitary matrices.
This section is organized as follows.

\begin{itemize}

\item
In Subsection 2.1,
we generalize the Riemann type relations suggested in [b], which induce  ``half formulas'' for the Jacobi or Riemann theta functions.
The results obtained in this subsection are valid for any imaginary quadratic fields $K$.

\item
The Cartan matrices of root systems are positive-definite matrices which play significant roles in various areas in mathematics.
In Subsection 2.2,
we give a simple expression of the theta function $\Theta^{A_h}$
defined by the use of the Cartan matrix $A_h$ of rank $h$ of the root system of type $A$.
The results obtained in this subsection are valid for any imaginary quadratic fields $K$ also.

\end{itemize}

We remark that
$\mathbb{Q}\left(\sqrt{-3}\right)$ ($\mathbb{Q}\left(\sqrt{-1}\right)$, resp.) is the imaginary quadratic field
with the smallest (second smallest, resp.) discriminant.
These imaginary quadratic fields are characterized by the property that those rings of integers have non-trivial units.
Furthermore,
theta functions (\ref{theta1}) for these fields are closely related to the moduli spaces of the concrete families of $K3$ surfaces as in Table 1 in Introduction.
Thus, theta functions for $\mathbb{Q}\left(\sqrt{-3}\right)$ and $\mathbb{Q}\left(\sqrt{-1}\right)$ are supported by solid motivations to study. 

\begin{itemize}

\item In Subsection 2.3, we give examples of cubic relations of theta functions for $K=\mathbb{Q}\left(\sqrt{-3}\right)$.

\item In Subsection 2.4, we give examples of quartic relations of theta functions for $K=\mathbb{Q}\left(\sqrt{-1}\right)$.

\end{itemize}

In this section,
in order to make the expressions of our results simple,
we will study formulas for the cases  $B_0=O$.
It should be reasonable for us to focus on such cases.
It is known that
many practically useful relations are given in terms of  theta functions for $B_0=O$,
like (\ref{RiemannQuad}), (\ref{Jacobihalf-angle}) or (\ref{doubleangle}) below.
Also, 
we are able to know how to apply Theorem \ref{ThmRiemannRel} by considering such cases.

\subsection{Particular quadratic theta relations}

For the Riemann theta function of (\ref{RiemannTheta}),
we set $\vartheta\begin{bmatrix}a \\ b \end{bmatrix}(\Omega)=\vartheta\begin{bmatrix}a \\ b \end{bmatrix}(0,\Omega)$ for $\Omega \in \mathfrak{S}_g$.
It is known that
they satisfy a quadratic equation
\begin{align}\label{RiemannQuad}
\vartheta\begin{bmatrix} a_1 \\ 0 \end{bmatrix} (\Omega )
\cdot
\vartheta\begin{bmatrix} a_2 \\ 0 \end{bmatrix} (\Omega )
=
\sum_{d\in (\mathbb{Z}/2\mathbb{Z})^g} 
\vartheta\begin{bmatrix} \frac{1}{2}(d+a_1+a_2) \\ 0 \end{bmatrix} (2\Omega )
\cdot
\vartheta\begin{bmatrix} \frac{1}{2}(d+a_1-a_2) \\ 0 \end{bmatrix} (2\Omega )
\end{align}
(see \cite{Mu} Chap.II Sect.6).
There are various interesting  applications of (\ref{RiemannQuad}).
For example, let us consider the  case of $g=1$ and let $\tau \in \mathbb{H}$ stands for $\Omega$.
Then, $\vartheta\begin{bmatrix}0\\0 \end{bmatrix}(\tau)$ ($\vartheta\begin{bmatrix}\frac{1}{2}\\0 \end{bmatrix}(\tau)$, resp.)  
is traditionally denoted by $\vartheta_{00}(\tau)$ ($\vartheta_{10}(\tau)$, resp.).
The relation (\ref{RiemannQuad}) derives ``half  formulas''
\begin{align}\label{Jacobihalf-angle}
 \vartheta_{00}^2 (\tau) = \vartheta_{00}^2 (2\tau) + \vartheta_{10}^2(2\tau), 
 \quad
  \vartheta_{10}^2 (\tau) = 2\vartheta_{00} (2\tau)  \vartheta_{10}(2\tau).
\end{align}
They are useful.
For instance, by  transformation formulas for $\tau\mapsto -\frac{1}{\tau}$ in \cite{Mu} Table V in Chap.I Sect.9,
the above formulas induce ``double formulas''
\begin{align}\label{doubleangle}
 \vartheta_{00}^2 (2\tau) = \frac{1}{2}(\vartheta_{00}^2 (\tau) + \vartheta_{01}^2(\tau) ), 
 \quad
  \vartheta_{01}^2 (2\tau) = \vartheta_{00} (\tau)  \vartheta_{01}(\tau). 
\end{align}
We remark that 
the formulas (\ref{doubleangle}) are very important to understand Gauss's approach to the arithmetic-geometric means and elliptic integrals (for example, see \cite{C}).

In this subsection,
we will give natural extensions of the above useful results for the Riemann theta functions.

We can regard the following proposition as a natural extension of (\ref{RiemannQuad}).

\begin{prop}\label{Prop01}
For $a\in K^g$,
let $\Theta\begin{bmatrix} a \\ 0 \end{bmatrix}(W) $ be the theta function (\ref{theta1}) associated with $K$.
Let $\alpha_1,\alpha_2\in K^g$.
Then, 
one has
\begin{align*}
&\Theta\begin{bmatrix} \alpha_1 \\ 0 \end{bmatrix} (|\delta|^2W)
\cdot \Theta\begin{bmatrix} \alpha_2 \\ 0 \end{bmatrix} (W) \notag\\
&=\sum_{\scriptsize{\begin{matrix} u\in (\mathbb{Z}/2|\delta|^2 \mathbb{Z})^g \\ v\in (\mathbb{Z}/2\mathbb{Z})^g \\ w\in (\mathbb{Z}/|\delta|^2\mathbb{Z})^g \end{matrix}}  }
\Theta\begin{bmatrix} \frac{1}{2\delta }\left(\alpha_1 \delta+\alpha_2+ u + v \delta+ 2w\right) \\0 \end{bmatrix}(2|\delta|^2W) 
\cdot
\Theta\begin{bmatrix} \frac{1}{2 \delta}\left(\alpha_1 \delta-\alpha_2+ u + v\delta \right) \\0 \end{bmatrix}(2|\delta|^2W) .
\end{align*}
\end{prop}

\begin{proof}
We note that the values of the theta functions in the right hand side of the formula are independent of the choice of representatives of $u,v,w$.
Let us consider the matrices
$$
T=\frac{1}{2\overline{\delta}} \begin{pmatrix}\overline{\delta} & 1 \\  \overline{\delta} &-1\end{pmatrix},  
\quad
 P= \begin{pmatrix} 2|\delta|^2 & 0 \\ 0 & 2|\delta|^2 \end{pmatrix}, 
 \quad
  Q=\begin{pmatrix} |\delta|^2 & 0 \\ 0 & 1 \end{pmatrix}.$$
Then, we have $T^{-1} = \begin{pmatrix} 1& 1 \\ \overline{\delta}& -\overline{\delta}\end{pmatrix} \in {\rm Mat}(2,2;\mathfrak{O}_K)$.
Therefore,
due to the definition of the group $G_2$ of Theorem \ref{ThmRiemannRel},
it follows that $G_2$ is the unit group: $G_2=\{0\}.$
From Theorem \ref{ThmRiemannRel}, we obtain
\begin{align}\label{SumBZ}
\Theta^{Q}\begin{bmatrix} A_0 {}^t\overline{T}^{-1} \\ O \end{bmatrix}(W)
=\sum_{A\in G_1}  \Theta^{P}\begin{bmatrix} A_0 +A  \\ O \end{bmatrix}(W).
\end{align}
Now,
for $m_1,m_2\in \mathfrak{O}_K^g$,
we have
$$
(m_1,m_2) {}^t\overline{T} = \left( \frac{m_1\delta+m_2}{2\delta} ,\frac{m_1\delta -m_2}{2\delta} \right) = \left( \frac{r+2m_2}{2\delta} ,\frac{r}{2\delta} \right), 
$$
where 
$r=m_1 \delta -m_2$.
Since  $r$ and $m_2$ attain  arbitrary elements of  $\mathfrak{O}_K^g= \left(\mathbb{Z} + \delta \mathbb{Z}\right)^g$, we obtain
$$
G_1=\left\{ \left( \frac{u+v\delta + 2w}{2\delta}, \frac{u+v\delta}{2\delta} \right) \hspace{1mm} \bigg| \hspace{1mm} u\in (\mathbb{Z}/2|\delta|^2\mathbb{Z})^g, v\in (\mathbb{Z}/2\mathbb{Z})^g,  w\in (\mathbb{Z}/|\delta|\mathbb{Z})^g \right\}.
$$  
If $A_0{}^t \overline{T} ^{-1} =(\alpha_1,\alpha_2)$, then $A_0=\left( \frac{\alpha_1\delta+\alpha_2}{2\delta} ,\frac{\alpha_1\delta -\alpha_2}{2\delta} \right).$
Hence, by (\ref{SumBZ}), together with (\ref{ProdTheta1}) and (\ref{ProdTheta2}),
we have the assertion.
\end{proof}

The next proposition gives another extension of (\ref{RiemannQuad}).

\begin{prop}\label{Prop11}
One has
\begin{align*}
&\Theta\begin{bmatrix} \alpha_1 \\ 0 \end{bmatrix} (W)
\cdot \Theta\begin{bmatrix} \alpha_2 \\ 0 \end{bmatrix} (W) \notag\\
&=\sum_{\scriptsize{\begin{matrix} u\in (\mathbb{Z}/2|\delta|^2 \mathbb{Z})^g \\ v\in (\mathbb{Z}/2\mathbb{Z})^g \\ w\in (\mathbb{Z}/|\delta|^2 \mathbb{Z})^g \end{matrix}}  }
\Theta\begin{bmatrix} \frac{1}{2\delta }\left(\alpha_1+\alpha_2+ u + v\delta+ 2w\right) \\0 \end{bmatrix}(2|\delta|^2W) 
\cdot
\Theta\begin{bmatrix} \frac{1}{2\delta}\left(\alpha_1-\alpha_2+ u+ v\delta\right) \\0 \end{bmatrix}(2|\delta|^2 W) .
\end{align*}
\end{prop}

\begin{proof}
Let us consider the matrices
$$
T=\frac{1}{2\overline{\delta}}\begin{pmatrix} 1 & 1 \\ 1 & -1 \end{pmatrix},  
\quad
P= \begin{pmatrix}2|\delta|^2  & 0 \\0 & 2|\delta|^2  \end{pmatrix},
\quad
Q=\begin{pmatrix} 1 & 0 \\ 0 &1 \end{pmatrix}.
$$ 
Then, we have $T^{-1}=\overline{\delta} \begin{pmatrix} 1 & 1 \\ 1 & -1 \end{pmatrix} \in {\rm Mat}(2,2;\mathfrak{O}_K)$.
By an argument similar to the proof of Proposition \ref{Prop01} (1), we have the formula in question.
In particular,
we can see that the finite groups $G_1$ and $G_2$ are the same as those of Proposition \ref{Prop01}.
\end{proof}

\subsection{Expressions of theta functions related with the root system $A_h$}

Let us denote the Cartan matrix of the root lattice of type $A_h$ by the letter expressing the type, i.e.,
$$
A_h
=\begin{pmatrix}
2& -1 & \cdots &0 &0\\
-1 & 2  &\cdots &0&0\\
\vdots & \vdots  & \ddots & \vdots&\vdots \\
0 &  0  & \cdots &2 &-1 \\
0 &  0  & \cdots &-1 &2 
\end{pmatrix}.
$$
In this subsection, we will give an expression of the theta function of (\ref{ThetaAB})
coming from the real positive-definite matrix $P=A_h$.

\begin{prop}
For $\alpha_1,\ldots,\alpha_h \in K^g$, 
let $\Theta^{A_h} 
\begin{bmatrix}
(\alpha_1,\alpha_2,\ldots,\alpha_h)\\
O
\end{bmatrix} 
(W)$
be
 the theta function (\ref{ThetaAB}) for $P=A_h$ associated with $K$.
Then,
one has the expression
\begin{align*}
&\Theta^{A_h} 
\begin{bmatrix}
(\alpha_1,\alpha_2,\ldots,\alpha_h)\\
O
\end{bmatrix} 
(W)\\
&=
\sum_{u_1,v_1} \cdots \sum_{u_h,v_h} 
\Theta
\begin{bmatrix}
 \left(\frac{\alpha_1+u_1 +  v_1 \delta}{h \delta} \right)\\
0
\end{bmatrix}
( (h+1) h |\delta|^2 W)
\cdot 
\Theta
\begin{bmatrix}
 \left(-\frac{\alpha_1+u_1 +  v_1 \delta}{h \delta} + \frac{\alpha_2 +u_2+ v_2 \delta}{(h-1)\delta} \right)\\
0
\end{bmatrix}
( h (h-1) |\delta|^2 W)\\
&\hspace{5cm}
\cdots
\Theta
\begin{bmatrix}
 \left(-\frac{\alpha_{h-1}+u_{h-1} +  v_{h-1} \delta }{2\delta} + \frac{\alpha_h +u_h}{\delta} \right)\\
0
\end{bmatrix}
(2|\delta|^2 W).
\end{align*}
Here, the summation $\sum_{u_1,v_1} \cdots \sum_{  u_h , v_h} $ extends over all $u_j\in (\mathbb{Z}/|\delta|^2(h+1-j) \mathbb{Z})^g$ and $v_j\in (\mathbb{Z}/(h+1-j) \mathbb{Z})^g$ for $j\in \{1,\ldots,h\}$.
\end{prop}

\begin{proof}
Set
\begin{align*}
T_h=
\begin{pmatrix}
\frac{1}{h} & 0 & 0 & \cdots & & 0 &0\\
-\frac{1}{h} &\frac{1}{h-1} & 0 & \cdots & & 0 & 0 \\
0 & -\frac{1}{h-1} & \frac{1}{h-2} & \cdots & & 0 & 0\\
0 & 0 & -\frac{1}{h-2}& \ddots & & 0&0   \\
\vdots &\vdots& &  & & \vdots &\vdots \\
0& 0& 0 & \cdots & & \frac{1}{2} & 0 \\
0& 0& 0 & \cdots & & -\frac{1}{2} & 1 \\
\end{pmatrix},
\quad
S_h
=\begin{pmatrix}
h& 0 & 0&\cdots &0&0 \\
h-1& h-1 &0 & \cdots &0&0\\
h-2& h-2 & h-2 & \cdots &0&0\\
\vdots & \vdots & \vdots &\ddots &\vdots& \vdots\\
2&2&2& \cdots &2 &0\\
1&1&1 & \cdots &1 &1
\end{pmatrix}
\end{align*}
and
\begin{align*}
P_h=
\begin{pmatrix}
(h+1) h& 0& \cdots &0\\
0& h(h-1) & \cdots &0 \\
\vdots& \vdots& \ddots&0 \\
0 & 0 & \cdots & 2\cdot1
\end{pmatrix}.
\end{align*}
Let us study the matrices
$$T=\frac{1}{\overline{\delta}} T_h, 
\quad
P=|\delta|^2 P_h,
\quad
Q=A_h.$$
Then, we have
$T^{-1}=\overline{\delta} S_h \in {\rm Mat}(2,2;\mathfrak{O}_K) $.
So,  the group $G_2$ of Theorem \ref{ThmRiemannRel} is equal to $\{0\}$.
Also, if we take $(m_1,\ldots, m_h) \in {\rm Mat}(g,h;\mathfrak{O}_K)$,
we have
\begin{align*}
(m_1,m_2,\ldots,m_h) {}^t \overline{T}=\frac{1}{\delta}\left( \frac{m_1}{h}, -\frac{m_1}{h} + \frac{m_2}{h-1}, -\frac{m_2}{h-1}+\frac{m_3}{h-2},\ldots, -\frac{m_{h-1}}{2}+m_h \right).
\end{align*}
Recall that $\mathfrak{O}_K = \mathbb{Z} + \delta \mathbb{Z}$.
Considering the fact that integers $k$ and $k+1$ are coprime if $k\geq 2$,
the group $G_1$ of Theorem \ref{ThmRiemannRel} is given by 
\begin{align*}
G_1&=\bigg\{\frac{1}{\delta}\left( \frac{u_1+ v_1 \delta}{h},-\frac{u_1+ v_1 \delta }{h} + \frac{u_2+ v_2 \delta}{h-1},\ldots, -\frac{u_{h-1}+ v_{h-1} \delta}{2}+u_h  \right) \\
&\hspace{5cm} \Big| \hspace{1mm}  u_j\in \mathbb{Z}/(|\delta|^2 (h+1-j) \mathbb{Z})^g, v_j \in (\mathbb{Z}/(h+1-j)\mathbb{Z})^g  \bigg\}.
\end{align*}
Moreover, if we set
$A_0 {}^t \overline{T}^{-1} =(\alpha_1,\alpha_2,\ldots,\alpha_h)$,
then
$A_0= \frac{1}{\delta}\left( \frac{\alpha_1}{h}, - \frac{\alpha_1}{h}+\frac{\alpha_2}{h-1},\ldots,-\frac{\alpha_{h-1}}{2}+\alpha_h \right) $.
Hence, we have the  equation in question.
\end{proof}

\subsection{Examples of cubic theta relations for $\mathbb{Q}\left(\sqrt{-3}\right)$}

In this subsection, we set $d=3$, $K=\mathbb{Q}\left(\sqrt{-3}\right)$, $\mathfrak{O}_K=\mathbb{Z}[\omega]$, where $\omega=\frac{-1 + \sqrt{-3}}{2}$.
Dern and Krieg  \cite{DK} study theta functions on $\mathbb{H}_I^{(2)}$ associated with the imaginary quadratic field $\mathbb{Q}\left(\sqrt{-3}\right)$.
Those theta functions are closely related to the period mapping for a family of lattice polarized $K3$ surfaces (see \cite{NS}).
Thus, theta functions (\ref{theta1}) associated with $K=\mathbb{Q}\left(\sqrt{-3}\right)$ are interesting objects.

\begin{prop}\label{PropR-3}
Let $K=\mathbb{Q}\left( \sqrt{-3} \right)$.
Take $\alpha_1,\alpha_2,\alpha_3 \in K^g$.
For $a\in K^g$,
let $\Theta\begin{bmatrix} a \\ 0 \end{bmatrix}(W) $ be the theta function (\ref{theta1}) associated with $K$.
Then, it holds
\begin{align}\label{Qubic-3}
&\Theta\begin{bmatrix} \alpha_1 \\ 0 \end{bmatrix} (W)\cdot
\Theta\begin{bmatrix} \alpha_2 \\ 0 \end{bmatrix} (W)\cdot
\Theta\begin{bmatrix} \alpha_3 \\ 0 \end{bmatrix} (W) \notag \\
&=
\sum_{
\scriptsize{
\begin{matrix} r \in (\mathfrak{O}_K/3\mathfrak{O}_K)^g \\ s\in (\sqrt{-3}\mathfrak{O}_K/3\mathfrak{O}_K)^g \end{matrix}}} 
\Theta\begin{bmatrix}    \frac{\alpha_1 + \omega^2 \alpha_2 + \omega \alpha_3 + r}{3} \\ 0 \end{bmatrix} (3W)
\cdot\Theta\begin{bmatrix}   \frac{\alpha_1 + \omega \alpha_2 + \omega^2 \alpha_3 + r+s }{3} \\ 0 \end{bmatrix} (3W)
\cdot
\Theta\begin{bmatrix}  \frac{\alpha_1 + \alpha_2 + \alpha_3 +r -s}{3}\\ 0 \end{bmatrix} (3W) .
\end{align}
\end{prop}

\begin{proof}
Take the matrices
$$T
=\frac{1}{3}
\begin{pmatrix}
1 & 1 & 1 \\
1 & \omega & \omega^2 \\
1 & \omega^2 & \omega\\
\end{pmatrix},
\quad
P=3I_3 ,
\quad
Q=I_3.
$$
Since $T^{-1} = 3{}^t \overline{T}=
\begin{pmatrix}
1 & 1 & 1 \\
1 & \omega^2 & \omega \\
1 & \omega & \omega^2\\
\end{pmatrix}
\in {\rm Mat}(3,3; \mathfrak{O}_K)$,
the finite group $G_2$ of Theorem \ref{ThmRiemannRel} is equal to $\{0\}$.
Also, for $(m_1,m_2,m_3) \in {\rm Mat}(g,3;\mathfrak{O}_K),$
we have
$
(m_1,m_2,m_3) {}^t \overline{T} =\frac{1}{3} (m_1+m_2+m_3, m_1+\omega^2 m_2 +\omega m_3, m_1+\omega m_2+\omega^2 m_3).
$
Now, $1+\omega + \omega^2 =0$.
Then,
setting $r=m_1+\omega^2 m_2 + \omega_1 m_3$ and $r'=m_1+\omega m_2 +\omega^2 m_3$,
we obtain 
\begin{align*}
m_1 + m_2 +m_3 =3m_1 +(-m_1-\omega^2 m_2 -\omega m_3) +(-m_1-\omega m_2 - \omega^2 m_3)\equiv -r -r'
\quad \quad (\text{mod } 3 \mathfrak{O}_K^g).
\end{align*}
Moreover, set $s=r'-r =\sqrt{-3} (m_2-m_3)$. 
Note that $m_2-m_3$ can take any values of $\mathfrak{O}_K^g$.
Therefore, the finite group $G_1$ of Theorem \ref{ThmRiemannRel} becomes
$$
G_1=\left\{\frac{1}{3}(-2r-s,r,r+s)  \hspace{1mm} \Big| \hspace{1mm}  r\in (\mathfrak{O}_K/3\mathfrak{O}_K)^g, s\in(\sqrt{-3}\mathfrak{O}_K/3\mathfrak{O}_K)^g  \right\}.
$$
If we set
$A_0 {}^t \overline{T}^{-1} =(\alpha_1, \alpha_2, \alpha_3)\in {\rm Mat}(g,3; \frac{1}{3}\mathfrak{O}_K/\mathfrak{O}_K)$, 
we obtain
$A_0=\frac{1}{3}(\alpha_1 + \alpha_2 + \alpha_3, \alpha_1+\omega^2 \alpha_2+\omega \alpha_3, \alpha_1+\omega \alpha_2+\omega^2 \alpha_3).$
Since $-2 r \equiv r $ $(\text{mod } 3\mathfrak{O}_K)$ for any $r \in \mathfrak{O}_K$,
 we have the formula in question.
\end{proof}

If we take particular $\alpha_1,\alpha_2,\alpha_3 \in K^g$,
the equation
(\ref{Qubic-3}) becomes simple.
For example, by putting $ \alpha_1 = \alpha_2 = \alpha_3 ={}^t( v_1,\ldots, v_g) \in K^g$, we have the following equation:
\begin{align}\label{Cor1}
&\Theta ^3 \begin{bmatrix} \begin{pmatrix} v_1 \\  \vdots \\v_g \end{pmatrix} \\ 0 \end{bmatrix} (W)\notag\\
&=
\sum_{
\scriptsize{
\begin{matrix} r_1,\ldots,r_g \in \mathfrak{O}_K/3\mathfrak{O}_K \\ s_1,\ldots, s_g\in \sqrt{-3}\mathfrak{O}_K/3\mathfrak{O}_K \end{matrix}}  } 
\Theta \begin{bmatrix}    \frac{1}{3}\begin{pmatrix}r_1 \\ \vdots \\ r_g \end{pmatrix} \\ 0 \end{bmatrix} (3W)
\cdot\Theta \begin{bmatrix}   \frac{1}{3}\begin{pmatrix}r_1+s_1 \\  \vdots \\ r_g+s_g \end{pmatrix}  \\ 0 \end{bmatrix} (3W)
\cdot
\Theta\begin{bmatrix}  \begin{pmatrix} v_1\\ \vdots \\ v_g \end{pmatrix} +  
\frac{1}{3}\begin{pmatrix} r_1 -s_1 \\ \vdots \\ r_g -s_g \end{pmatrix} \\ 0 \end{bmatrix} (3W).
\end{align}
Additionally, by substituting
$\alpha_1 = {}^t( v_1,\ldots, v_g),$
$\alpha_2 =  {}^t( v_1,\ldots, v_g) +\frac{1}{\sqrt{-3}} {}^t(1,\ldots,1) $
and
$\alpha_3 =  {}^t( v_1,\ldots, v_g) -\frac{1}{\sqrt{-3}} {}^t(1,\ldots,1) $,
we have the following equation:
\begin{align}\label{Cor2}
& \Theta  \begin{bmatrix} \begin{pmatrix} v_1 \\  \vdots \\ v_g \end{pmatrix} \\ 0 \end{bmatrix} (W)
\cdot \Theta \begin{bmatrix} \begin{pmatrix} v_1\\  \vdots \\ v_g \end{pmatrix} +\frac{1}{\sqrt{-3}}  \begin{pmatrix} 1\\  \vdots \\ 1 \end{pmatrix}  \\ 0 \end{bmatrix} (W) 
\cdot \Theta \begin{bmatrix} \begin{pmatrix} v_1\\ \vdots \\ v_g \end{pmatrix} -\frac{1}{\sqrt{-3}}  \begin{pmatrix} 1\\  \vdots \\ 1 \end{pmatrix}  \\ 0 \end{bmatrix}  (W) \notag\\
&=
\sum_{
\scriptsize{
\begin{matrix} r_1,\ldots,r_g \in \mathfrak{O}_K/3\mathfrak{O}_K \\ s_1,\ldots, s_g\in \sqrt{-3}\mathfrak{O}_K/3\mathfrak{O}_K \end{matrix}}  } 
\Theta\begin{bmatrix} \frac{1}{3}\begin{pmatrix} -1 \\ \vdots \\ -1 \end{pmatrix} +  \frac{1}{3}\begin{pmatrix} r_1 \\ \cdots \\ r_g \end{pmatrix} \\ 0 \end{bmatrix} (3W) 
\cdot \Theta\begin{bmatrix} \frac{1}{3}\begin{pmatrix} 1 \\ \vdots \\ 1 \end{pmatrix} + \frac{1}{3}\begin{pmatrix} r_1+s_1 \\ \vdots \\ r_g+s_g \end{pmatrix} \\ 0 \end{bmatrix} (3W) \notag \\
&\hspace{10cm}
\cdot  \Theta\begin{bmatrix} \begin{pmatrix} v_1 \\ \vdots \\ v_g \end{pmatrix} + \frac{1}{3}\begin{pmatrix} r_1 -s_1 \\ \vdots \\ r_g - s_g \end{pmatrix} \\ 0 \end{bmatrix} (3W).
\end{align}

As mentioned previously,
the Riemann theta functions (\ref{RiemannTheta}) whose characteristics are in $\frac{1}{2}\mathbb{Z}/\mathbb{Z}$
have ample amount of applications (for example,  see  the beginning of Section 2.1).
Similarly,
it seems natural to study theta function of Proposition \ref{PropR-3}  
whose characteristics are in $\frac{1}{3}\mathfrak{O}_K/\mathfrak{O}_K.$
Remark that
$
 \mathfrak{O}_K /3\mathfrak{O}_K
=\left\{0,\pm1, \pm \sqrt{-3}, \pm 1 \pm \sqrt{-3}\right\} \simeq (\mathbb{Z}/3\mathbb{Z}) \oplus (\mathbb{Z} /3\mathbb{Z})$
and
$
\sqrt{-3} \mathfrak{O}_K /3\mathfrak{O}_K
=\left\{0, \pm \sqrt{-3}\right\} \simeq (\mathbb{Z}/3\mathbb{Z}).
$
Now,
let $\rho_{j,k} \in \mathbb{Z}/3\mathbb{Z}=\{0,1,-1\}$ 
for $j\in \{1,\ldots, g\} , k\in \{1,2\}$.
We introduce a theta function labeled by $\frac{1}{3}\mathbb{Z}/\mathbb{Z}$:
\begin{align}\label{ThetaBmat}
\Theta\begin{Bmatrix}
\rho_{11},& \rho_{12}\\
\vdots & \vdots \\
\rho_{g1},&\rho_{g2}
\end{Bmatrix}
(W)
=\Theta\begin{bmatrix}
\begin{pmatrix}
\frac{1}{3} \rho_{11} + \frac{1}{\sqrt{-3}} \rho_{12}\\
\vdots\\
\frac{1}{3} \rho_{g1} + \frac{1}{\sqrt{-3}} \rho_{g2}
\end{pmatrix}
\\
0
\end{bmatrix}
(W).
\end{align}
The function (\ref{ThetaBmat}) simplifies the relations (\ref{Qubic-3}),
if $\alpha_1,\alpha_2, \alpha_3 \in \frac{1}{3}\mathfrak{O}_K^g$.
In fact, the all characteristics appearing in (\ref{Qubic-3}) are closed in $\frac{1}{3} \mathfrak{O}_K/\mathfrak{O}_K$.
For example,
in the simplest case such that $v_1=\cdots=v_g=0$, 
the equation (\ref{Cor1}) becomes 
\begin{align*}
&\left(\Theta\begin{Bmatrix}
0, & 0 \\
\vdots & \vdots \\
0, & 0 
\end{Bmatrix}(W)\right)^3
=
\sum_{\rho_{11},\rho_{12},\ldots,\rho_{g1},\rho_{g2}\in \{0,1,-1\}} 
\sum_{\sigma_{12},\ldots,\sigma_{g2} \in \{0,1,-1\}}\\
&\hspace{5cm}
\Theta\begin{Bmatrix} 
\rho_{11},& \rho_{12}\\
\vdots & \vdots \\
\rho_{g1},&\rho_{g2}
 \end{Bmatrix}(3W)
 \cdot
 \Theta\begin{Bmatrix} 
\rho_{11},& \rho_{12}+\sigma_{12}\\
\vdots & \vdots \\
\rho_{g1},&\rho_{g2}+\sigma_{g2}
 \end{Bmatrix}(3W)
  \cdot
 \Theta\begin{Bmatrix} 
\rho_{11},&\rho_{12}- \sigma_{12}\\
\vdots & \vdots \\
\rho_{g1},&\rho_{g2}-\sigma_{g2}
 \end{Bmatrix}(3W).
\end{align*}
Similarly,
if we put $v_1=\cdots=v_g=0$,
the equation (\ref{Cor2}) becomes 
\begin{align*}
&\Theta\begin{Bmatrix}
0, & 0 \\
\vdots & \vdots \\
0, & 0 
\end{Bmatrix}(W)
\cdot
\Theta\begin{Bmatrix}
0, & 1 \\
\vdots & \vdots \\
0, & 1 
\end{Bmatrix}(W)
\cdot
\Theta\begin{Bmatrix}
0, & -1 \\
\vdots & \vdots \\
0, & -1 
\end{Bmatrix}(W)\\
&=
\sum_{\rho_{11},\rho_{12},\ldots,\rho_{g1},\rho_{g2}\in \{0,1,-1\}} 
\sum_{\sigma_{12},\ldots,\sigma_{g2} \in \{0,1,-1\}}\\
&\hspace{2.5cm}
\Theta\begin{Bmatrix} 
\rho_{11}-1,& \rho_{12}\\
\vdots & \vdots \\
\rho_{g1}-1,&\rho_{g2}
 \end{Bmatrix}(3W)
 \cdot
 \Theta\begin{Bmatrix} 
\rho_{11}+1,& \rho_{12}+\sigma_{12}\\
\vdots & \vdots \\
\rho_{g1}+1,&\rho_{g2}+\sigma_{g2}
 \end{Bmatrix}(3W)
  \cdot
 \Theta\begin{Bmatrix} 
\rho_{11},&\rho_{12}- \sigma_{12}\\
\vdots & \vdots \\
\rho_{g1},&\rho_{g2}-\sigma_{g2}
 \end{Bmatrix}(3W).
\end{align*}
Although the above relations can be expressed in simpler forms
via  (\ref{ThetaRed}),
we will not pursue them in this paper.

\subsection{Examples of quartic theta relations for $\mathbb{Q}\left(\sqrt{-1}\right)$ }

In this subsection, we set $d=1$, $K=\mathbb{Q}\left(\sqrt{-1}\right)$, $\mathfrak{O}_K=\mathbb{Z}[i]$, where $i=\sqrt{-1}$.
Let us study the theta functions (\ref{ThetaAB}) associated with $\mathbb{Q}\left( \sqrt{-1}\right)$.
Since  $\mathfrak{O}_K$  is just the ring of Gaussian integers,
there are already various research for such theta functions (see \cite{F}, \cite{Mat}, \cite{MT}).
We give the following theta relation as an application of Theorem \ref{ThmRiemannRel}.

\begin{prop}\label{PropG}
Let $K=\mathbb{Q}\left( \sqrt{-1}\right)$. 
Take $\alpha_1,\alpha_2,\alpha_3,\alpha_4 \in K^g$.
For $a\in K^g$,
let $\Theta\begin{bmatrix} a \\ 0 \end{bmatrix}(W) $ be the theta function (\ref{theta1}) associated with $K$.
Then, it holds
\begin{align*}
&\Theta\begin{bmatrix} \alpha_1 \\ 0 \end{bmatrix} (W)\cdot
\Theta\begin{bmatrix} \alpha_2 \\ 0 \end{bmatrix} (W)\cdot
\Theta\begin{bmatrix} \alpha_3 \\ 0 \end{bmatrix} (W)\cdot 
\Theta\begin{bmatrix} \alpha_4 \\ 0 \end{bmatrix} (W) \notag \\
&=
\sum_{
\scriptsize{
\begin{matrix} r \in (\mathfrak{O}_K/4\mathfrak{O}_K)^g \\ s_1,s_2 \in (2\mathfrak{O}_K/4\mathfrak{O}_K)^g \end{matrix}}} 
\Theta\begin{bmatrix}  \frac{\alpha_1 -i \alpha_2 -i \alpha_3 -\alpha_4 +r}{4}\\ 0 \end{bmatrix} (4W) 
\cdot\Theta\begin{bmatrix}  \frac{-i\alpha_1 +  \alpha_2 -  \alpha_3-i \alpha_4 -i r +s_1}{4} \\ 0 \end{bmatrix} (4W)\notag\\
&\hspace{5cm}\cdot\Theta\begin{bmatrix}   \frac{i \alpha_1 +  \alpha_2  -  \alpha_3 +i \alpha_4 -i r+s_2 }{4} \\ 0 \end{bmatrix} (4W)
\cdot\Theta\begin{bmatrix} \frac{\alpha_1 +i\alpha_2 +i\alpha_3 -\alpha_4 -r- i s_1 +i s_2 }{4} \\ 0 \end{bmatrix} (4W).
\end{align*}
\end{prop}

\begin{proof}
Let us consider the matrices 
$$
T=\frac{1}{4} \begin{pmatrix} 1 & i & i & -1 \\ 
 i & 1& -1& i  \\ 
 -i & 1& -1& -i \\ 
 1 & -i& -i& -1  
 \end{pmatrix},
 \quad
 P=4I_4,
 \quad
 Q=I_4.
 $$
 Then, we have $T^{-1} = 4{}^t \overline{T} 
 =\begin{pmatrix}
 1 & -i & i & 1 \\
 -i & 1 &1 & i \\
 -i & -1 & -1 &i \\
-1 & -i & i & -1
 \end{pmatrix} \in {\rm Mat } (4,4; \mathfrak{O}_K).$
 Thereby, the group $G_2$ of Theorem \ref{ThmRiemannRel} is equal to $\{0\}.$ 
 Letting $(m_1,m_2,m_3,m_4)\in {\rm Mat}(g,4;\mathfrak{O}_K)$, 
 we have
 $(m_1,m_2,m_3,m_4){}^t \overline{T} 
 =\frac{1}{4} (r,r',r'',r'''), $ 
 where
 \begin{align*}
 \begin{cases}
 & r = m_1 -i m_2 -i m_3 - m_4,\\ 
 & r' = -i m_1 +m_2 - m_3 -i m_4,\\
 & r'' = i m_1 +m_2 -m_3 +i m_4,\\
 & r''' = m_1 + i m_2 + i m_3 -m_4.
 \end{cases}
 \end{align*}
Then, $i r'-r =2(im_2 +m_4) $ and $-i r'' + r =2(m_1 -im_2)$ hold. 
Here, $im_2 +m_4$ and  $m_1 -im_2$ can take any values of $\mathfrak{O}_K^g$ respectively.
Also, 
$$
r+i r' -i r'' +r''' =4m_1 \equiv 0 \quad (\text{mod } 4\mathfrak{O}_K^g).
$$
Therefore, the finite group $G_1$ of Theorem \ref{ThmRiemannRel} has an expression
$$
G_1=\left\{\frac{1}{4}(r,-ir +s_1, -ir +s_2 , -r -s_1 i +s_2 i) \hspace{1mm} \bigg| \hspace{1mm} r\in (\mathfrak{O}_K/4\mathfrak{O}_K)^g, s_1,s_2\in (2\mathfrak{O}_K/4\mathfrak{O}_K)^g  \right\}.
$$
Setting $A_0 {}^t \overline{T}^{-1} =(\alpha_1,\alpha_2,\alpha_3,\alpha_4)$, 
we obtain $A_0=\frac{1}{4}(\alpha_1 -i \alpha_2 -i \alpha_3 - \alpha_4, -i \alpha_1 +\alpha_2 - \alpha_3 -i \alpha_4, i \alpha_1 +\alpha_2 -\alpha_3 +i \alpha_4, \alpha_1 + i \alpha_2 + i \alpha_3 -\alpha_4)$.
By summarizing the above, we have the formula in question. 
\end{proof}

Note that
$
\mathfrak{O}_K/4\mathfrak{O}_K 
=\{\rho_1 + \rho_2 i \mid  \rho_1,\rho_2 \in \mathbb{Z}/4\mathbb{Z}\}.
$
Now, we set
\begin{align}\label{ThetaBmat1}
\Theta \begin{Bmatrix}
\rho_{11}, & \rho_{12}\\
\vdots & \vdots \\
\rho_{g1}, & \rho_{g2}\\
\end{Bmatrix}
(W)
=
\Theta \begin{bmatrix}
\begin{pmatrix}
\frac{1}{4} \rho_{11} + \frac{i}{4} \rho_{12}\\
\vdots \\
\frac{1}{4} \rho_{g1} + \frac{i}{4} \rho_{g2}\\
\end{pmatrix}
\\
0 
\end{bmatrix}(W),
\end{align}
for $\rho_{j,k} \in \mathbb{Z}/4\mathbb{Z} =\{0,1,2,3\}$ for $j\in \{1,\ldots,g\}, k\in \{1,2\}$. 
The expression (\ref{ThetaBmat1}) enables us to obtain  simple forms of the equation of the Proposition \ref{PropG}, 
if $\alpha_1,\alpha_2, \alpha_3,\alpha_4 \in \frac{1}{4} \mathfrak{O}_K^g$.
For example, if we put 
$\alpha_1=\alpha_2=\alpha_3=\alpha_4=0$,
we obtain an equation
\begin{align*}
&\left(\Theta\begin{Bmatrix}
0, &0 \\ 
\vdots & \vdots \\
 0, &0 
\end{Bmatrix}
(W) \right)^4\\
&=
\sum_{\rho_{11},\rho_{12},\ldots,\rho_{g1},\rho_{g2}\in \{0,\pm1,2\}} 
\sum_{\scriptsize
\begin{matrix}\sigma_{11}, \sigma_{12},\ldots,\sigma_{g1},\sigma_{g2} \in \{0,2\}\\
\sigma'_{11}, \sigma'_{12},\ldots,\sigma'_{g1},\sigma'_{g2} \in \{0,2\}\end{matrix}
}\\
&\hspace{1.5cm}
\Theta\begin{Bmatrix}\rho_{11}, & \rho_{12}\\
\vdots & \vdots \\
\rho_{g1}, & \rho_{g2}  \end{Bmatrix} (4W) 
\cdot 
\Theta\begin{Bmatrix}\rho_{12} + \sigma_{11}, & -\rho_{11} +\sigma_{12}\\
\vdots & \vdots \\
\rho_{g2}+\sigma_{g1}, & -\rho_{g1} +\sigma_{g2}  \end{Bmatrix} (4W)
\cdot
\Theta\begin{Bmatrix}\rho_{12} + \sigma'_{11}, & -\rho_{11} +\sigma'_{12}\\
\vdots & \vdots \\
\rho_{g2}+\sigma'_{g1}, & -\rho_{g1} +\sigma'_{g2}  \end{Bmatrix} (4W) \\
&\hspace{8.8cm}
\cdot
\Theta\begin{Bmatrix}-\rho_{11} +\sigma_{12} -\sigma_{12}', & -\rho_{12}-\sigma_{11}+\sigma_{11}'\\
\vdots & \vdots \\
-\rho_{g1}+\sigma_{g2}-\sigma_{g2}' , & -\rho_{g2} -\sigma_{g1} +\sigma_{g1}'  \end{Bmatrix} (4W) .
\end{align*}

\section*{Acknowledgment}
The  author is supported by JSPS Grant-in-Aid for Scientific Research (22K03226), JST FOREST Program (JPMJFR2235) and  MEXT LEADER.
He is thankful to Professor Hironori Shiga for valuable discussions.
He also appreciates the reviewer's accurate suggestions for improvement of  the manuscript.

\vspace{1mm}
\begin{center}
\hspace{8.8cm}\textit{Atsuhira  Nagano}\\
\hspace{8.8cm}\textit{Faculty of Mathematics and Physics,}\\
 \hspace{8.8cm} \textit{Institute of Science and Engineering,}\\
\hspace{8.8cm}\textit{Kanazawa University,}\\
\hspace{8.8cm}\textit{Kakuma, Kanazawa, Ishikawa}\\
\hspace{8.8cm}\textit{920-1192, Japan}\\
 \hspace{8.8cm}\textit{(E-mail: atsuhira.nagano@gmail.com)}
  \end{center}

\end{document}